\documentclass[12pt]{amsart}

\usepackage{amsfonts,amssymb}
\usepackage{hyperref, graphicx}
\usepackage{latexsym}
\usepackage{amsmath,amsthm}
\usepackage[all,cmtip]{xy}
\newtheorem{theorem}{Theorem}[section]
\newtheorem{proposition}[theorem]{Proposition}
\newtheorem{lemma}[theorem]{Lemma}
\newtheorem{corollary}[theorem]{Corollary}

\theoremstyle{definition}
\newtheorem{definition}[theorem]{Definition}

\numberwithin{equation}{section}

\newcommand{\C}{\mathbb{C}}
\newcommand{\N}{\mathbb{N}}
\newcommand{\Z}{\mathbb{Z}}
\newcommand{\EC}{\mathcal{C}^\mathrm{EC}}

\newcommand{\coker}{\operatorname{coker}}

\begin{document}
\baselineskip=15.5pt

\title[Semi-finite vector bundles with connection]{On the semi-finite vector bundles with connection over K\"ahler manifolds}

\author{Sanjay Amrutiya}
\address{Department of Mathematics, IIT Gandhinagar,
 Near Village Palaj, Gandhinagar - 382355, India}
 \email{samrutiya@iitgn.ac.in}

\author{Indranil Biswas}

\address{Department of Mathematics, Shiv Nadar University, NH91, Tehsil
Dadri, Greater Noida, Uttar Pradesh 201314, India}

\email{indranil.biswas@snu.edu.in, indranil29@gmail.com}



\date{}

\begin{abstract}
Let $X$ be a compact connected K\"ahler manifold. We consider the category 
$\EC(X)$ of flat holomorphic connections $(E,\, \nabla^E)$ over $X$ satisfying the condition that the underlying
holomorphic vector bundle $E$ admits a filtration of holomorphic subbundles preserved by the connection 
$\nabla^E$ such that the monodromy of the induced connection on each successive quotient has finite image.
The category $\EC(X)$, equipped with the neutral fiber functor that 
sends any object $(E,\, \nabla^E)$ to the fiber $E_{x_0}$, where $x_0\, \in\, X$ is a
fixed point, defines a neutral Tannakian 
category over $\C$. Let $\varpi^{\mathrm{EC}}(X,\, x_0)$ denote the affine group scheme 
corresponding to this neutral Tannakian category $\EC(X)$. Let $\pi^{\mathrm{EN}}(X,\, x_0)$ be an 
extension of the Nori fundamental group scheme over $\C$ \cite{Ot}. 

We show that $\pi^{\mathrm{EN}}(X,\, x_0)$ is a closed subgroup scheme of $\varpi^{\mathrm{EC}}(X,\, x_0)$.
Finally, we discuss an example illustrating that if $X$ is not K\"ahler, then the natural homomorphism
$\pi^{\mathrm{EN}}(X,\, x_0)\, \longrightarrow\, \varpi^{\mathrm{EC}}(X,\, x_0)$ might fail to be an embedding.
\end{abstract}
\maketitle

\section{Introduction}\label{intro}

Tannaka duality has played a key role in the development of various generalizations 
of fundamental groups, designed to capture the geometry of schemes and complex 
manifolds (see \cite{DM} for the Tannakian category). A seminal example of such a 
construction is the foundational work of M. Nori \cite{No}.

Recall that a vector bundle $E$ over a scheme $X$ is said to be finite if there are 
two distinct polynomials $f$ and $g$ with non-negative integer coefficients such that 
$f(E)$ is isomorphic to $g(E)$ (see Definition \ref{def-finite}).
Over the field of complex numbers, it follows from \cite{No} that a vector bundle 
$E$ is finite if and only if it is given by a representation of the fundamental 
group $\pi_1(X)$ that factors through a finite group.

Let $X$ be a compact connected K\"ahler manifold, and let $E$ be a holomorphic vector 
bundle over $X$. Using the results from \cite{DPS} and \cite{Si}, it is shown in 
\cite{BHS} that $E$ admits a flat holomorphic connection with finite monodromy group
if and only if  $E$ is finite vector bundle over $X$.

Otabe \cite{Ot} introduced a broader category than that of finite bundles in 
characteristic zero. He calls a vector bundle $E$ over $X$ to be semi-finite 
if it admits a filtration of subbundles such that each successive quotient is a finite bundle 
(see Definition \ref{def-semifinite}). The category of semi-finite bundles over $X$, 
equipped with the neutral fiber functor, forms a neutral Tannakian category over the base field $k$. 
The corresponding affine group scheme is called the extended Nori fundamental group 
scheme and it is denoted by $\pi^{\mathrm{EN}}(X,\,x)$. 

C. Simpson \cite{Si} considered the category of all flat holomorphic connections 
$(E,\, \nabla^E)$ over a compact connected K\"ahler manifold $X$. 
This category is equipped with the operators of direct sum, tensor product, and 
dualization, and it is a neutral Tannakian category with the usual fiber functor given 
by a base point $x_0$ (see \cite[p. 70]{Si}). The resulting fundamental group scheme 
over $\C$ is denoted by $\varpi(X,\, x_0)$, which represents a pro-algebraic completion 
of the topological fundamental group. 

We consider the category $\EC(X)$ of all flat holomorphic connections $(E,\, \nabla^E)$ over 
$X$, where the underlying holomorphic vector bundle $E$ admits a filtration of subbundles 
preserved by the connection $\nabla^E$, such that the induced connection on each 
successive quotient has finite monodromy (see Section \ref{sec-2}). The category 
$\EC(X)$, equipped with the neutral fiber functor that sends any object $(E,\, \nabla^E)$ 
to the fiber $E_{x_0}$, defines a neutral Tannakian category over $\C$. 
Let $\varpi^{\mathrm{EC}}(X,\, x_0)$ denote the affine group scheme corresponding to 
this Tannakian category. There is a natural faithfully flat homomorphism:
$$
\varpi(X,\, x_0)\ \longrightarrow\ \varpi^{\mathrm{EC}}(X,\, x_0)
$$
of affine group schemes over $\C$ (see Theorem \ref{thm-fgsk}). 

Next, we define a forgetful functor 
$$
\xi \,:\, \EC(X) \,\longrightarrow\, \mathcal{C}^{\mathrm{EN}}(X)
$$
that sends any pair $(E,\, \nabla^E)$ to the underlying vector bundle $E$. This induces
a morphism of affine group schemes over $\C$:
$$
\xi^* \ :\ \pi^{\mathrm{EN}}(X,\,x_0) \ \longrightarrow\ \varpi^{\mathrm{EC}}(X,\, x_0)\,.
$$
We prove the following result (see Theorem \ref{main-thm}):

\begin{theorem}
The induced morphism 
$\xi^* \ :\ \pi^{{\rm EN}}(X,\,x_0) \ \longrightarrow\ \varpi^{\rm EC}(X,\, x_0)$ 
is a closed immersion.
\end{theorem}

Finally, we discuss an example illustrating that if $X$ is not K\"ahler, the above theorem 
does not hold in general.

\section{Semi--finite bundles and flat connections}\label{sec-2}

Let $X$ be a compact connected K\"ahler manifold. The holomorphic tangent and cotangent bundles
of $X$ will respectively be denoted by $TX$ and $\Omega_X^1$.

Let $f(t) = \sum_{k=0}^n a_k t^k\,\in\, \N[t]$ be a polynomial
whose coefficients are nonnegative integers. For any holomorphic vector bundle
$E$ over $X$, define the holomorphic vector bundle
$$
f(E)\ :=\ \bigoplus_{k=0}^n (E^{\otimes k})^{\oplus a_k},
$$
where $E^{\otimes 0}$ denotes the trivial line bundle ${\mathcal O}_X$.

\begin{definition}[{\cite{No}}]\label{def-finite}
A holomorphic vector bundle $E$ over $X$ is said to be finite if there are two distinct polynomials $f$ and $g$,
whose coefficients are nonnegative integers, such that $f(E)$ is holomorphically isomorphic to $g(E)$.
\end{definition} 

A flat connection $\nabla$ on a holomorphic vector bundle $E$ on $X$ is called a holomorphic flat connection
if the sheaf of flat sections of $E$ for $\nabla$ is contained in the sheaf of holomorphic sections.

A holomorphic vector bundle $E$ on $X$ is finite if and only if it admits a flat
holomorphic connection for which the monodromy group is finite \cite{No} (see also
\cite{BHS}).

\begin{definition}[{\cite{Ot}}]\label{def-semifinite}
A holomorphic vector bundle $E$ on $X$ is called \textit{semi--finite} if it has a filtration of holomorphic subbundles 
$$
E\,=\,E_0\,\supset\, E_1 \,\supset\, E_2 \,\supset\,\cdots\, \supset\, E_{n+1}\,=\,0
$$ such that $E_i/E_{i+1}$ is a finite bundle, for all $0\, \leq\, i\, \leq\, n$.
\end{definition} 

Semi--finite vector bundles form a Tannakian category \cite[Proposition 2.14]{Ot}.
Let $\mathcal{C}^{\mathrm{EN}}(X)$ be the Tannakian category of semi--finite vector
bundles over $X$. Fix a base point $x_0 \,\in \,X$. Then, we get a neutral fiber functor on
$\mathcal{C}^{\mathrm{EN}}(X)$ that sends any $E$ to its
fiber $E_{x_0}$. The corresponding affine group scheme is denoted by $\pi^{\mathrm{EN}}(X,\,x_0)$ 
(see \cite{Ot} for more details).

\subsection{Flat connections}

Let $\mathbf{Conn}(X)$ be the category of all pairs $(E,\,\nabla^E)$, where $E$ is a 
holomorphic vector bundle on $X$ and $\nabla^E\,:\, E \,\longrightarrow
\, E \otimes_{\mathcal{O}_X} \Omega_X^1$ is a flat holomorphic connection on $E$.
A morphism $f\,:\, (E,\, \nabla^E) \,\longrightarrow\, (F,\, \nabla^F)$ in $\mathbf{Conn}(X)$
is a morphism of vector bundles $f_0\, :\, E\, \longrightarrow\, F$ that intertwines $\nabla^E$ 
and $\nabla^F$, meaning
$$
\nabla^F\circ f_0 \ =\ (f_0\otimes {\rm Id}_{\Omega_X^1})\circ\nabla^E.
$$

One can perform the usual operations between the objects in $\mathbf{Conn}(X)$ as follows:
Take two objects $(E,\,\nabla^E)$ and $(F, \,\nabla^F)$ in $\mathbf{Conn}(X)$.
\begin{itemize}
\item Direct Sum: $\nabla^{E\oplus F}\,:\, E\oplus F\,\longrightarrow \,(E \oplus F)\otimes_{\mathcal{O}_X} 
\Omega_X^1$ given by $\nabla^{E\oplus F}(s\oplus t)\,=\, \nabla^E(s) + \nabla^F(t)$, where 
$s$ and $t$ are locally defied holomorphic sections of $E$ and $F$ respectively.

\item Tensor Product: $\nabla^{E\otimes F}\,:\, E\otimes F\, \longrightarrow\, (E \otimes F)
\otimes_{\mathcal{O}_X} \Omega_X^1$ given by $\nabla^{E\otimes F}(s\otimes t)\,=\,
\nabla^E(s)\otimes t + s\otimes \nabla^F(t)$.

\item Dual: $\nabla^{E^*}\,:\, E^* \,\longrightarrow\, E^* \otimes_{\mathcal{O}_X} \Omega_X^1$ given by
$\nabla^{E^*}(\sigma)(s) = d(\sigma(s)) - \sigma(\nabla^E(s))$, where $\sigma$ (respectively, $s$) is a locally
defined holomorphic section of $E^*$ (respectively, $E$).
\end{itemize} 

Note that the category $\mathbf{Conn}(X)$ is an abelian $\C$-linear rigid tensor category with the
above operations. Fix a base point $x_0 \,\in\, X$. The category $\mathbf{Conn}(X)$, equipped with the
faithful fiber functor $T_{x_0} \,:\, \mathbf{Conn}(X) \,\longrightarrow\, \mathrm{Vec}(\C)$
that sends any object $(E,\, \nabla^E)$ to the fiber $E_{x_0}$, defines a neutral Tannakian category
over $\C$. By Tannaka duality, we get a pro-algebraic affine group scheme, which we shall denote
by
\begin{equation}\label{vp}
\varpi(X,\, x_0).
\end{equation}
See \cite{Si} for other equivalent description of $\varpi(X,\, x_0)$.

Let $\EC(X)$ be the full subcategory of $\mathbf{Conn}(X)$ consisting of all objects
$(E,\,\nabla^E)$ satisfying the the following condition:
\begin{itemize}
\item The holomorphic vector bundle $E$ admits a filtration of holomorphic subbundles
$$
E\,=\,E_0\,\supseteq\, E_1 \,\supseteq E_2\, \supseteq\,\cdots\, \supseteq\, E_{n}\,
\supseteq\, E_{n+1}\,=\, 0
$$
such that $\nabla^E(E_i) \,\subseteq\, E_i \otimes_{\mathcal{O}_X} \Omega_X^1$ for all 
$1\, \leq\, i\, \leq\, n$,  and the flat holomorphic connection on each quotient
$E_i/E_{i+1}$, $0\, \leq\, i\, \leq\, n$, induced by $\nabla^E$ has finite monodromy group.
\end{itemize}

\begin{lemma}\label{lemma-ext}
Let $0\,\longrightarrow\, (E_1,\, \nabla^{E_1})\,\stackrel{f}{\longrightarrow}\,
(E, \,\nabla^E)\, \stackrel{g}{\longrightarrow}\, (E_2,\, \nabla^{E_2})\,\longrightarrow\, 0$
be an exact sequence in $\mathbf{Conn}(X)$, meaning the underlying sequence of holomorphic 
vector bundles
is exact with $\nabla^E\big\vert_{E_1} \,=\, \nabla^{E_1}$ and $\nabla^{E_2}$ coincides
with the connection on $E_2$ induced by $\nabla^E$.
If $(E_1,\, \nabla^{E_1})$ and $(E_2, \nabla^{E_2})$ are objects of $\mathrm{\EC(X)}$, 
then $(E, \,\nabla^E)$ is also an object of $\EC(X)$.
\end{lemma}

\begin{proof}
Assume that $(E_1,\, \nabla^{E_1})$ and $(E_2,\, \nabla^{E_2})$ are objects of $\EC(X)$. Let
$$
E_1\,=\,V_0\,\supseteq\, V_1 \, \supseteq\,\cdots\, \supseteq\, V_{m}\,
\supseteq\, V_{m+1}\,=\, 0
$$
be a filtration of holomorphic subbundles satisfying the condition in the definition of
$\EC(X)$. Let
$$
E_2\,=\,W_0\,\supseteq\, W_1 \, \supseteq\,\cdots\, \supseteq\, W_{n}\,
\supseteq\, W_{n+1}\,=\, 0
$$
be a filtration of holomorphic subbundles satisfying the condition in the definition of
$\EC(X)$. Consider the following filtration of holomorphic subbundles of $E$:
$$
E\, \supseteq\, g^{-1}_0(W_1) \, \supseteq\,\cdots\, \supseteq\, g^{-1}_0(W_n)\,
\supseteq\, f_0(V_1) \, \supseteq\,\cdots\, \supseteq\, f_0(V_m) \, \supseteq \, 0.
$$
It clearly satisfies the condition in the definition of $\EC(X)$. Hence, 
$(E, \,\nabla^E)$ is an object of $\EC(X)$.
\end{proof}

\begin{lemma}\label{epi-lemma}
Let $f\,:\, (E,\, \nabla^E) \,\longrightarrow\, (F,\, \nabla^F)$ be an epimorphism in 
$\mathbf{Conn}(X)$. If $(E,\, \nabla^E)$ is an object of $\EC(X)$, then $(F,\, \nabla^F)$ 
is also an object of $\EC(X)$.
\end{lemma}

\begin{proof}
Since $(E,\, \nabla^E)$ is an object of $\EC(X)$, there exists a filtration
\begin{equation}\label{eq:1}
E\,=\,E_0\, \supseteq\, E_1\, \supseteq\, E_2\, \supseteq\, \cdots\,\supseteq\, E_n\,\supseteq\, 
E_{n+1} \,=\, 0
\end{equation}
with $\nabla^E(E_i)\, \subseteq\, E_i \otimes_{\mathcal{O}_X} \Omega_X^1$ for all $1\, \leq\, i\,
\leq\, n$, and the flat holomorphic connection on each quotient
$E_i/E_{i+1}$, $0\, \leq\, i\, \leq\, n$, induced by $\nabla^E$ has finite monodromy group.
Note that $\ker(f_0)$ equipped with the flat holomorphic connection $\nabla^{\ker(f_0)}$ 
induced by $\nabla^E$ is an object of $\mathbf{Conn}(X)$. Since $f$ is an epimorphism in 
$\mathbf{Conn}(X)$, the connection $\nabla^F$ coincides with the connection on 
$E/\ker(f_0)\,\simeq\, F$ induced by $\nabla^E$. Therefore, the filtration \eqref{eq:1} 
yields a filtration
$$
F\,=\,f_0(E_0)\, \supseteq\, f_0(E_1)\, \supseteq\, f_0(E_2)\, \supseteq\, \cdots\,\supseteq\,
f_0(E_m)\,\supseteq\, f_0(E_{m+1}) \,=\, 0
$$
with $\nabla^F(f_0(E_i))\, \subseteq\, f_0(E_i) \otimes_{\mathcal{O}_X} \Omega_X^1$, and the flat 
holomorphic connection on each quotient $f_0(E_i)/f_0(E_{i+1})$, $0\, 
\leq\, i\, \leq\, m$, induced by $\nabla^F$ has finite monodromy group.
\end{proof}

\begin{lemma}\label{lem-t}
For any object $(E,\, \nabla^E)$ of $\EC(X)$, the dual vector bundle $E^*$ equipped with the
connection induced by $\nabla^E$ is also an object of $\EC(X)$.
\end{lemma}

\begin{proof}
Let 
$$
E\,=\,E_0\, \supseteq\, E_1\, \supseteq\, E_2\, \supseteq\, \cdots\,\supseteq\, E_n\,\supseteq\, 
E_{n+1} \,=\, 0
$$
be a filtration of holomorphic subbundles such that $\nabla^E(E_i)\, \subseteq\, E_i
\otimes_{\mathcal{O}_X} \Omega_X^1$ for all $1\, \leq\, i\,
\leq\, n$, and the flat holomorphic connection on each quotient
$E_i/E_{i+1}$, $0\, \leq\, i\, \leq\, n$, induced by $\nabla^E$ has finite monodromy group.
Consider the dual filtration of $E^*$ given by the following surjections:
$$
E^*\,=\,E^*_0\, \longrightarrow\, E^*_1\, \longrightarrow\, \cdots\,\longrightarrow \, E^*_n
\,\longrightarrow\, E^*_{n+1} \,=\, 0.
$$
It is preserved by the connection on $E^*$ induced by $\nabla^E$. The monodromy of
a dual connection is the dual representation of the fundamental group. Hence, $E^*$ equipped 
with the connection induced by $\nabla^E$ is an object of $\EC(X)$.
\end{proof}

\begin{corollary}\label{cor1}
Let $f\,:\, (E,\, \nabla^E) \,\longrightarrow\, (F,\, \nabla^F)$ be an epimorphism in 
$\EC(X)$. Then $\ker(f_0)$ equipped with the connection induced by $\nabla^E$ is
also an object of $\EC(X)$.
\end{corollary}

\begin{proof}
Consider the dual $E^*\, \longrightarrow\, \ker (f_0)^*$ of the inclusion map
$\ker(f_0) \, \hookrightarrow\, E$. The connection of $E^*$ induced by $\nabla^E$ produces a
connection of $\ker(f_0)^*$. Indeed, this follows immediately from the fact that the epimorphism
$f_0$ is connection preserving. The vector bundle $E^*$ equipped with the connection induced by
$\nabla^E$ is an object of $\EC(X)$ (see Lemma \ref{lem-t}). Therefore, from Lemma \ref{epi-lemma},
it follows that $\ker(f_0)^*$ equipped with the above connection is an object of $\EC(X)$. 
Hence, from Lemma \ref{lem-t}, it follows that $\ker(f_0)$ equipped with the connection induced 
by $\nabla^E$ is an object of $\EC(X)$.
\end{proof}

\begin{proposition}\label{prop-EC-rigid}
The full subcategory $\EC(X)$ of $\mathbf{Conn}(X)$ is a $\C$-linear rigid tensor subcategory. 
\end{proposition}

\begin{proof}
We will show that $\EC(X)$ is a tensor category. Take two objects $(E,\, \nabla^E)$ and
$(F,\, \nabla^F)$ in $\EC(X)$. We need to show that $E \otimes F$ together with the induced 
connection $\nabla^{E\otimes F})$ is an object of $\EC(X)$. 
Since $(E,\, \nabla^E)$ is an object in $\EC(X)$, there exist a filtration
\begin{equation}\label{eq:2}
E\,=\,E_0\,\supseteq\, E_1\, \supseteq\, E_2\, \supseteq\, \cdots\, \supseteq\, E_n\,\supseteq\, 
E_{n+1}\, =\, 0
\end{equation}
with $\nabla^E(E_i)\, \subseteq\, E_i \otimes_{\mathcal{O}_X} \Omega_X^1$, and the flat holomorphic
connection on each quotient $E_i/E_{i+1}$, $0\, \leq\, i\, \leq\, n$, induced by $\nabla^E$ 
has finite monodromy group. Tensoring \eqref{eq:2} with $F$, we get a filtration 
$$
E\otimes F\,=\,E_0\otimes F\,\supseteq\, E_1\otimes F\, \supseteq\, \cdots\, \supseteq\,
E_n\otimes F\,\supseteq\, E_{n+1}\otimes F\, =\, 0
$$ 
with $(E_i \otimes F)/(E_{i+1} \otimes F)\, \simeq\, (E_i/E_{i+1})\otimes F$ and it is compatible 
with the induced connection $\nabla^{E\otimes F}$. By Lemma \ref{lemma-ext}, the 
category $\EC(X)$ is closed under taking extension in $\mathbf{Conn}(X)$. Consequently, it follows 
that $(E \otimes F,\, \nabla^{E\otimes F})$ is an object of $\EC(X)$. 

In Lemma \ref{lem-t}, it was shown that $\EC(X)$ is closed under taking duals. This completes the proof.
\end{proof}

\begin{proposition}\label{prop-EC-abelian}
The full subcategory $\EC(X)$ of $\mathbf{Conn}(X)$ is an abelian subcategory.
\end{proposition}
\begin{proof}
Let $f\,:\, (E,\, \nabla^E) \longrightarrow (F,\, \nabla^F)$ be a morphism in $\EC(X)$. 
Then, the image of $f_0$, denoted by $\mathrm{Im}(f_0)$, with the flat holomorphic connection 
$\nabla^{\mathrm{Im}(f_0)}$ induced by $\nabla^F$ is a subobject of $(F,\, \nabla^F)$ in 
$\mathbf{Conn}(X)$. Hence, we have an epimorphism 
\begin{equation}\label{eq-epi}
f\,:\, (E,\, \nabla^E) \longrightarrow (\mathrm{Im}(f_0),\, \nabla^{\mathrm{Im}(f_0)})
\end{equation}
in $\mathbf{Conn}(X)$. By Lemma \ref{epi-lemma}, it follows that 
$(\mathrm{Im}(f_0),\, \nabla^{\mathrm{Im}(f_0)})$ is an object of $\EC(X)$. Again from 
Lemma \ref{epi-lemma}, we can conclude that $\coker(f_0)$ equipped with the connection 
induced by $\nabla^F$ is an object of $\EC(X)$. 

Since $(\mathrm{Im}(f_0),\, \nabla^{\mathrm{Im}(f_0)})$ is an object of $\EC(X)$, 
the morphism in \eqref{eq-epi} is an epimorphism in $\EC(X)$. Hence, using Corollary 
\ref{cor1}, it follows that $\ker(f_0)$ equipped with the connection induced by $\nabla^E$ 
is an object of $\EC(X)$.
\end{proof}

\begin{theorem}\label{thm-fgsk}
The category $\EC(X)$, equipped with the faithful fiber functor $T_{x_0}$ that sends any object 
$(E,\, \nabla^E)$ to the fiber $E_{x_0}$, defines a neutral Tannakian category
over $\C$. Denoting by $\varpi^{\mathrm{EC}}(X,\, x_0)$ the affine group scheme corresponding 
to the neutral Tannakian category $\EC(X)$, there is a natural faithfully flat homomorphism
$$
\varpi(X,\, x_0)\ \longrightarrow\ \varpi^{\mathrm{EC}}(X,\, x_0)
$$
of affine group schemes over $\C$ (see \eqref{vp} for $\varpi(X,\, x_0)$).
\end{theorem}

\begin{proof}
The theorem follows from Proposition \ref{prop-EC-rigid}, Proposition \ref{prop-EC-abelian}
and \cite[Proposition 2.21]{DM}.
\end{proof}

Let $\xi \,:\, \EC(X) \,\longrightarrow\, \mathcal{C}^{\mathrm{EN}}(X)$ be the forgetful functor
that sends any pair $(E,\, \nabla^E)$ to $E$. Then, we get an induced morphism 
\begin{equation}\label{eq:3}
\xi^* \ :\ \pi^{\mathrm{EN}}(X,\,x_0) \ \longrightarrow\ \varpi^{\mathrm{EC}}(X,\, x_0)
\end{equation}
of affine group scheme over $\C$.

\begin{theorem}\label{main-thm}
The morphism $\xi^*$ in \eqref{eq:3} of affine group schemes over $\C$ is a closed immersion.
\end{theorem}

\begin{proof}
Take a semi-finite vector bundle $E$ over $X$. Let
\begin{equation}\label{j1}
E\,=\,E_0\, \supseteq\, E_1\, \supseteq\, E_2\, \supseteq\, \cdots\, \supseteq\, E_\ell\,
\supseteq\, E_{\ell+1}\, =\, 0
\end{equation}
be a filtration of holomorphic subbundles such that the quotient $E_i/E_{i+1}$ is a finite vector
bundle for all $0\, \leq\, i\, \leq\, \ell$. Therefore, each $E_i/E_{i+1}$ admits a flat holomorphic
connection whose monodromy is a finite group \cite[Theorem 2.3]{BHS}. In particular, $E_i/E_{i+1}$
is polystable and $c_j(E_i/E_{i+1})\,=\, 0$ for all $j\, \geq\, 1$. In fact, $E_i/E_{i+1}$ admits a
unique flat unitary connection and this connection has finite monodromy (see \cite{BHS}). 

Since each successive quotient $E_i/E_{i+1}$ in \eqref{j1} is polystable and $c_j(E_i/E_{i+1})\,=\, 0$
for all $j\, \geq\, 1$, by \cite[Corollary 3.10 and its following remark]{Si}, the vector bundle $E$ admits
a unique flat holomorphic connection --- denote it by $\nabla^E$ --- such that the following hold:
\begin{itemize}
\item $\nabla^E(E_i)\, \subseteq\, E_i \otimes_{\mathcal{O}_X} \Omega_X^1$ for every
$1\, \leq\, i\, \leq\, \ell$, and

\item the flat holomorphic connection on $E_i/E_{i+1}$ induced by $\nabla^E$ coincides with the
unique unitary flat connection on $E_i/E_{i+1}$, for all $0\, \leq\, i\, \leq\, \ell$.
\end{itemize}
 
Hence $(E,\, \nabla^E)$ is an object of $\EC(X)$. The theorem follows from this, using
\cite[Proposition 2.21]{DM}.
\end{proof}

\subsection*{Example}

Let $M$ be a compact complex manifold. Then, $\EC(M)$ and $\mathcal{C}^{\mathrm{EN}}(M)$ 
can be defined as before, and both form neutral Tannakian categories, once a base point $x_0$ is chosen. We have the homomorphism
$$
\xi^* \,:\, \pi^{\mathrm{EN}}(M,\,x_0) \,\longrightarrow\, \varpi^{\mathrm{EC}}(M,\, x_0)
$$
as in \eqref{eq:3} constructed using the forgetful map $(E,\, \nabla^E) \, \longmapsto\, E$.

However, if $M$ is not K\"ahler, then Theorem \ref{main-thm} is not true in general. 
To construct an example, let $T \,=\, \C/(\Z + \tau\Z)$ be an elliptic curve,
where $\tau\,\in\,  \C$ with $\mathrm{Im}(\tau) \,>\, 0$. Take
$$M\, =\, \mathbb{S}^{2m+1}\times \mathbb{S}^{2n+1},\ \ \ m,\, n\, \geq\, 1,$$
with a complex structure such that we have a 
holomorphic principal $T$--bundle
$$M\ \longrightarrow\ \C\mathbb{P}^m \times \C\mathbb{P}^n$$
(see \cite{Ho} for the construction of such a holomorphic principal $T$--bundle).

Then, $M$ is not K\"ahler because $H^2(M,\, {\mathbb R})\,=\, 0$.
It is known that $h^{1, 0}(M) \,=\, 1$ (see \cite[p.~232]{Ho} or \cite{Bo}).
Let $\eta$ be a non-zero element in $H^1(X,\, \mathcal{O}_X)$. Then, $\eta$ gives an 
extension $E$ of ${\mathcal O}_M$ by ${\mathcal O}_M$
$$
0\,\longrightarrow\, \mathcal{O}_M \,\longrightarrow\, E \,\longrightarrow\, \mathcal{O}_M
\,\longrightarrow\, 0
$$
which does not split holomorphically.
Hence, $E$ is a non-trivial unipotent holomorphic rank 2 vector bundle over $M$.
Since $M$ is simply connected, it follows that $E$ does not admit any
nontrivial flat holomorphic connection. This show that the morphism $\xi^*$
is not closed immersion for this $M$.
 
\section*{Acknowledgements}
SA is supported by the SERB-DST grant(CRG/2023/000477). IB is supported by a J. C. Bose 
Fellowship (JBR/2023/000003).

\section*{Mandatory declarations}

The authors have no conflict of interest to declare that are relevant to this article.
No data were generated or used.

%%%%%%%%%%%%%%%%%%%%%%%%%%%%%%%%%%%%%%%%%%%%%%%%%%%%%%%%%%%%%%%%%%%%%%%

\end{document}